\documentclass[12pt,english]{article}
\usepackage{color}
\usepackage{babel}
\usepackage{units}
\usepackage{mathrsfs}
\usepackage{amsmath}
\usepackage{amsthm}
\usepackage{amssymb}
\usepackage[unicode=true,
 bookmarks=true,bookmarksnumbered=true,bookmarksopen=false,
 breaklinks=false,pdfborder={0 0 1},backref=false,colorlinks=true]
 {hyperref}
\hypersetup{pdftitle={Fractional kinetic in spatial ecological model},
 pdfauthor={{Jos{\'e}Lu{\'\i}sdaSilva},
 pdfsubject={General Fractional calculus and asymptotic of traveling waves},
 pdfkeywords={Bolker-Pacala model, General fractional derivative, Distributed order derivative, Subordination principle},
 },linkcolor=blue,citecolor=blue,urlcolor=blue,filecolor=blue}

\makeatletter
\theoremstyle{plain}
\newtheorem{thm}{\protect\theoremname}
\theoremstyle{remark}
\newtheorem{rem}[thm]{\protect\remarkname}
\theoremstyle{plain}
\newtheorem{prop}[thm]{\protect\propositionname}
\theoremstyle{plain}
\newtheorem{lem}[thm]{\protect\lemmaname}
\theoremstyle{definition}
\newtheorem{example}[thm]{\protect\examplename}
\theoremstyle{definition}
\newtheorem{defn}[thm]{\protect\definitionname}

\usepackage{babel}

\usepackage{babel}
\usepackage{mathrsfs}
\numberwithin{equation}{section}

\providecommand{\definitionname}{Definition}
\providecommand{\examplename}{Example}
\providecommand{\lemmaname}{Lemma}
\providecommand{\propositionname}{Proposition}
\providecommand{\remarkname}{Remark}
\providecommand{\theoremname}{Theorem}

\makeatother

\providecommand{\definitionname}{Definition}
\providecommand{\examplename}{Example}
\providecommand{\lemmaname}{Lemma}
\providecommand{\propositionname}{Proposition}
\providecommand{\remarkname}{Remark}
\providecommand{\theoremname}{Theorem}

\begin{document}
\title{From Random Times to Fractional Kinetics}
\author{\textbf{Anatoly N.~Kochubei}\\
 Institute of Mathematics,\\
 National Academy of Sciences of Ukraine, \\
 Tereshchenkivska 3, \\
 Kyiv, 01601 Ukraine\\
 Email: kochubei@imath.kiev.ua\and \textbf{Yuri Kondratiev}\\
 Department of Mathematics, University of Bielefeld, \\
 D-33615 Bielefeld, Germany,\\
 Dragomanov University, Kiev, Ukraine,\\
 Email: kondrat@math.uni-bielefeld.de\and\textbf{Jos{\'e} Lu{\'\i}s
da Silva},\\
 CIMA, University of Madeira, Campus da Penteada,\\
 9020-105 Funchal, Portugal.\\
 Email: joses@staff.uma.pt}
\maketitle
\begin{abstract}
In this paper we study the effect of the subordination by a general
random time-change to the solution of a model on spatial ecology in
terms of its evolution density. In particular on traveling waves for
a non-local spatial logistic equation. We study the Cesaro limit of
the subordinated dynamics in a number of particular cases related
to the considered fractional derivative making use of the Karamata-Tauberian
theorem.

\end{abstract}

\section{Introduction}

A study of a random time change in a Markov process $X_{t}$ was initiated
by S.\ Bochner \cite{Bochner1955} by considering an independent
Markov random time $\xi_{t}$. The resulting process $Y_{t}=X_{\xi_{t}}$
is again Markov and is called a subordinated process. In a pioneering
work \cite{Kolsrud}, T.\ Kolsrud initiated the study of the general
independent random time process. Later the concept of a random time
change became an effective tool in the study of physical phenomena
related to relaxation and diffusion problems in complex systems. We
refer here to the section "Historical notes"
in \cite{Mura08}. An additional essential motivation for the random
time change did appear in applications to biological models. The point
is such that there exists a notion of \emph{biological time} specific
for each particular type of biological system and which is very different
compared to the usual time scale employed in physics. One of the possibilities
to incorporate this notion is related to a random time change. Moreover,
this approach gives the chance to include in the model an effective
influence of dynamical random environment in which our system in located.

An especially interesting situation appears for the case of an inverse
subordinator $\xi_{t}$. We will describe this framework roughly and
leaving the details and more precise formulations in Subsection\ \ref{subsec:Ecological-model}
below. The marginal distributions $\mu_{t}$ of a Markov process $X_{t}$
describe an evolution of states in the considered systems and deliver
an essential information for the study of the dynamics. We call that
the statistical dynamics in contrast to the stochastic dynamics $X_{t}$
which contains more detailed information about the evolution of the
system. The statistical dynamics may be formulated by means of the
Fokker-Planck-Kolmogorov (FPK) evolution equation (weak sense) 
\[
\frac{\partial\mu_{t}}{\partial t}=L^{\ast}\mu_{t},
\]
where $L^{\ast}$ is the (dual) generator of the Markov process on
states. For an inverse subordinator $\xi_{t}$ and the process $Y_{t}=X_{\xi_{t}}$,
denote $\nu_{t}$ the corresponding marginal distributions. The key
observation is such that the dynamics of $\nu_{t}$ is described by
the evolution equation 
\[
\mathbb{D}_{t}^{\xi}\nu_{t}=L^{\ast}\nu_{t},
\]
where $\mathbb{D}_{t}^{\xi}$ denotes a generalized (convolutional)
fractional derivative in time canonically associated with $\xi_{t}$.
This fractional Fokker-Planck-Kolmogorov equation (FFPK) gives the
main technical instrument for the study of subordinated statistical
dynamics. There is a well known particular case of the inverse to
stable subordinators. In this case, such standard objects appear as
Caputo-Djrbashian fractional derivatives, and all the well developed
techniques of fractional calculus work perfectly. But in the case
of general inverse subordinators we should think about proper subclasses,
for which certain analytic properties of the related objects may be
established, see e.g.\ \cite{Chen2018}.

Note that there exist two possible points of view. We can start from
an inverse subordinator and arrive in FFPK equation \cite{Meerschaert2006,Meerschaert2008,Kolokoltsov2009,Toaldo2015}.
Or, vice versa, we develop at first a notion of generalized fractional
derivative and then search for a probabilistic interpretation of the
solution. The latter was first realized in \cite{Kochubei11}. For
a detailed discussion of both possibilities see \cite{Chen2018}.

The aim of this paper is to analyze the effects of random time changes
on Markov dynamics for certain models of interacting particle systems
in the continuum. For the concreteness, we will consider the important
Bolker-Pacala model in the spatial ecology that is a particular case
of general birth-and-death processes in the continuum. The scheme
of our study is the following. The FPK equation for the states of
the model may be reformulated in terms of a hierarchical evolution
of correlation functions. In a kinetic scaling limit this system of
equations leads to a kinetic hierarchy for correlation functions and
to a non-linear evolution equation for the density of the system.
The latter is a Vlasov-type non-linear and non-local evolution equation.
Considering a random time change by an inverse subordinator we arrive
in the FFPK equation for correlation functions and to a fractional
kinetic hierarchy. For a discussion of this approach and certain new
properties of fractional kinetic hierarchy see \cite{KocKon2017}.

A surprising feature appearing in the kinetic hierarchy is related
to the dynamics of the density of the considered systems. The evolution
of the density in the time changed kinetics is not the solution of
a related Vlasov-type equation with a fractional time derivative as
one may expect. In reality, this dynamics is a subordination of the
solution to kinetic equation for the hierarchy in the initial physical
time. A particular problem which does appear in this situation is
related with an effect of the subordination for such special solutions
as traveling waves which are known for the Bolker-Pacala model. In
the special case of stable subordinators this question was studied
in \cite{Da Silva2018}. In the present paper we are dealing with
certain classes of inverse subordinators for which the analysis of
subordinated waves may be carried out.

We summarize our observation as follows: a heuristic consideration
of a kinetic equation for the density with a fractional time derivative
has no relation to the real dynamics in the kinetic limit of the time
changed Markov evolution of the model. The correct behavior is given
by a subordination of the solution to the kinetic equation in physical
time.

\section{Preliminaries}

\subsection{General Facts and Notation}

Let $\mathcal{B}(\mathbb{R}^{d})$ be the family of all Borel sets
in $\mathbb{R}^{d}$, $d\geq1$ and let $\mathcal{B}_{b}(\mathbb{R}^{d})$
denote the system of all bounded sets in $\mathcal{B}(\mathbb{R}^{d})$.

The space of $n$-point configurations in an arbitrary $Y\in\mathcal{B}(\mathbb{R}^{d})$
is defined by 
\[
\Gamma^{(n)}(Y):=\big\{\eta\subset Y\big|\;|\eta|=n\big\},\quad n\in\mathbb{N},
\]
where $|\cdot|$ the cardinality of a finite set. We also set $\Gamma^{(0)}(Y):=\{\emptyset\}$.
As a set, $\Gamma^{(n)}(Y)$ may be identified with the symmetrization
of 
\[
\widetilde{Y^{n}}=\big\{(x_{1},\ldots,x_{n})\in Y^{n}\big|x_{k}\neq x_{l}\text{ if }k\neq l\big\}.
\]

The configuration space over the space $\mathbb{R}^{d}$ consists
of all locally finite subsets (configurations) of $\mathbb{R}^{d}$,
namely, 
\begin{equation}
\Gamma=\Gamma(\mathbb{R}^{d}):=\big\{\gamma\subset\mathbb{R}^{d}\,\big|\,|\gamma\cap\Lambda|<\infty,\text{ for all }\Lambda\in\mathcal{B}_{b}(\mathbb{R}^{d})\big\}.\label{eq:conf_space}
\end{equation}
The space $\Gamma$ is equipped with the vague topology, i.e., the
minimal topology for which all mappings $\Gamma\ni\gamma\mapsto\sum_{x\in\gamma}f(x)\in\mathbb{R}$
are continuous for any continuous function $f$ on $\mathbb{R}^{d}$
with compact support. Note that the summation in $\sum_{x\in\gamma}f(x)$
is taken over only finitely many points of $\gamma$ belonging to
the support of $f$. It was shown in \cite{Kondratiev2006} that with
the vague topology $\Gamma$ may be metrizable and it becomes a Polish
space (i.e., a complete separable metric space). Corresponding to
this topology, the Borel $\sigma$-algebra $\mathcal{B}(\Gamma)$
is the smallest $\sigma$-algebra for which all mappings 
\[
\Gamma\ni\gamma\mapsto|\gamma_{\Lambda}|\in\mathbb{N}_{0}:=\mathbb{N}\cup\{0\}
\]
are measurable for any $\Lambda\in\mathcal{B}_{b}(\mathbb{R}^{d})$.
Here $\gamma_{\Lambda}:=\gamma\cap\Lambda$.

It follows that one can introduce the corresponding Borel $\sigma$-algebra
on $\Gamma^{(n)}(Y)$, which we denote by $\mathcal{B}(\Gamma^{(n)}(Y))$.
The space of finite configurations in an arbitrary $Y\in\mathcal{B}(\mathbb{R}^{d})$
is defined by 
\[
\Gamma_{0}(Y):=\bigsqcup_{n\in\mathbb{N}_{0}}\Gamma^{(n)}(Y).
\]
This space is equipped with the topology of disjoint unions. Therefore
one can introduce the corresponding Borel $\sigma$-algebra $\mathcal{B}(\Gamma_{0}(Y))$.
In the case of $Y=\mathbb{R}^{d}$ we will omit $Y$ in the notation,
thus $\Gamma_{0}:=\Gamma_{0}(\mathbb{R}^{d})$ and $\Gamma^{(n)}:=\Gamma^{(n)}(\mathbb{R}^{d})$.

The restriction of the Lebesgue product measure $(dx)^{n}$ to $\bigl(\Gamma^{(n)},\mathcal{B}(\Gamma^{(n)})\bigr)$
will be denoted by $m^{(n)}$, and we set $m^{(0)}:=\delta_{\{\emptyset\}}$.
The Lebesgue--Poisson measure $\lambda$ on $\Gamma_{0}$ is defined
by 
\begin{equation}
\lambda:=\sum_{n=0}^{\infty}\frac{1}{n!}m^{(n)}.\label{eq:lp_measure}
\end{equation}
For any $\Lambda\in\mathcal{B}_{b}(\mathbb{R}^{d})$, the restriction
of $\lambda$ to $\Gamma(\Lambda):=\Gamma_{0}(\Lambda)$ will be also
denoted by $\lambda$. The space $\bigl(\Gamma,\mathcal{B}(\Gamma)\bigr)$
is the projective limit of the family of spaces $\big\{(\Gamma(\Lambda),\mathcal{B}(\Gamma(\Lambda)))\big\}_{\Lambda\in\mathcal{B}_{b}(\mathbb{R}^{d})}$.
The Poisson measure $\pi$ on $\bigl(\Gamma,\mathcal{B}(\Gamma)\bigr)$
is given as the projective limit of the family of measures $\{\pi^{\Lambda}\}_{\Lambda\in\mathcal{B}_{b}(\mathbb{R}^{d})}$,
where $\pi^{\Lambda}:=e^{-m(\Lambda)}\lambda$ is the probability
measure on $\bigl(\Gamma(\Lambda),\mathcal{B}(\Gamma(\Lambda))\bigr)$.
Here $m(\Lambda)$ is the Lebesgue measure of $\Lambda\in\mathcal{B}_{b}(\mathbb{R}^{d})$.

For any measurable function $f:\mathbb{R}^{d}\rightarrow\mathbb{R}$
we define the \emph{Lebesgue--Poisson exponent} 
\begin{equation}
e_{\lambda}(f,\eta):=\prod_{x\in\eta}f(x),\quad\eta\in\Gamma_{0};\qquad e_{\lambda}(f,\emptyset):=1.\label{eq:lp-exponent}
\end{equation}
Then, by (\ref{eq:lp_measure}), for $f\in L^{1}(\mathbb{R}^{d},dx)$
we obtain $e_{\lambda}(f)\in L^{1}(\Gamma_{0},d\lambda)$ and 
\begin{equation}
\int_{\Gamma_{0}}e_{\lambda}(f,\eta)\,d\lambda(\eta)=\exp\left(\int_{\mathbb{R}^{d}}f(x)\,dx\right).\label{LP-exp-mean}
\end{equation}

A set $M\in\mathcal{B}(\Gamma_{0})$ is called bounded if there exists
$\Lambda\in\mathcal{B}_{b}(\mathbb{R}^{d}))$ and $N\in\mathbb{N}$
such that $M\subset\bigsqcup_{n=0}^{N}\Gamma^{(n)}(\Lambda)$. We
will make use of the following classes of functions on $\Gamma_{0}$:
(i) $L_{\mathrm{ls}}^{0}(\Gamma_{0})$ is the set of all measurable
functions on $\Gamma_{0}$ which have local support, i.e., $H\in L_{ls}^{0}(\Gamma_{0})$,
if there exists $\Lambda\in\mathcal{B}_{b}(\mathbb{R}^{d})$ such
that $H\upharpoonright_{\Gamma_{0}\setminus\Gamma(\Lambda)}=0$, while
(ii) $B_{bs}(\Gamma_{0})$ is the set of bounded measurable functions
with bounded support, i.e., $H\upharpoonright_{\Gamma_{0}\setminus B}=0$
for some bounded $B\in\mathcal{B}(\Gamma_{0})$.

In fact, any $\mathcal{B}(\Gamma_{0})$-measurable function $H$ on
$\Gamma_{0}$ is a sequence of functions $\bigl\{ H^{(n)}\bigr\}_{n\in\mathbb{N}_{0}}$,
where $H^{(n)}$ is a $\mathcal{B}(\Gamma^{(n)})$-measurable function
on $\Gamma^{(n)}$.

On $\Gamma$ we consider the set of cylinder functions $\mathcal{F}_{cyl}(\Gamma)$.
These functions are characterized by the relation $F(\gamma)=F\upharpoonright_{\Gamma_{\Lambda}}(\gamma_{\Lambda})$.

The following mapping from $L_{ls}^{0}(\Gamma_{0})$ into $\mathcal{F}_{cyl}(\Gamma)$
plays a key role in our further considerations: 
\begin{equation}
KH(\gamma):=\sum_{\eta\Subset\gamma}H(\eta),\quad\gamma\in\Gamma,\label{eq:k-transform}
\end{equation}
where $H\in L_{ls}^{0}(\Gamma_{0})$. See, for example, \cite{KK99}
and references therein for more details. The summation in (\ref{eq:k-transform})
is taken over all finite sub-configurations $\eta\in\Gamma_{0}$ of
the (infinite) configuration $\gamma\in\Gamma$; this relationship
is represented symbolically by $\eta\Subset\gamma$. The mapping $K$
is linear, positivity preserving, and invertible, with 
\begin{equation}
K^{-1}F(\eta):=\sum_{\xi\subset\eta}(-1)^{|\eta\setminus\xi|}F(\xi),\quad\eta\in\Gamma_{0}.\label{k-1trans}
\end{equation}
Here and in the sequel, inclusions like $\xi\subset\eta$ hold for
$\xi=\emptyset$ as well as for $\xi=\eta$. We denote the restriction
of $K$ onto functions on $\Gamma_{0}$ by $K_{0}$.

A probability measure $\mu\in\mathcal{M}_{fm}^{1}(\Gamma)$ is called
locally absolutely continuous with respect to (w.r.t.) a Poisson measure
$\pi$ if for any $\Lambda\in\mathcal{B}_{b}(\mathbb{R}^{d})$ the
projection of $\mu$ onto $\Gamma(\Lambda)$ is absolutely continuous
w.r.t.~the projection of $\pi$ onto $\Gamma(\Lambda)$. By \cite{KK99},
there exists in this case a \emph{correlation functional} $\varkappa_{\mu}:\Gamma_{0}\rightarrow\mathbb{R}_{+}$
such that the following equality holds for any $H\in B_{bs}(\Gamma_{0})$:
\begin{equation}
\int_{\Gamma}(KH)(\gamma)\,d\mu(\gamma)=\int_{\Gamma_{0}}H(\eta)\varkappa_{\mu}(\eta)\,d\lambda(\eta).\label{eq:corr_funct}
\end{equation}
Restrictions $\varkappa_{\mu}^{(n)}$ of this functional on $\Gamma_{0}^{(n)}$,
$n\in\mathbb{N}_{0}$, are called \emph{correlation functions} of
the measure $\mu$. Note that $\varkappa_{\mu}^{(0)}=1$.

\subsection{Microscopic Spatial Ecological Model}

\label{subsec:Ecological-model}Let us consider the spatial ecological
model a.k.a.\ the Bolker-Pacala model, for the introduction and detailed
study of this model see \cite{BP1997,FM2004,FKK10,FKK12,FKKK15}.
Below we formulate certain results from these papers concerning the
Markov dynamics and mesoscopic scaling in the Bolker-Pacala model.

The heuristic generator $L$ in this model is defined on a space of
functions over the configuration space

\begin{align}
(LF)(\gamma)= & \sum_{x\in\gamma}\Bigl(m+\sum_{y\in\gamma\setminus x}a^{-}(x-y)\Bigr)[F(\gamma\setminus x)-F(\gamma)]\nonumber \\
 & +\sum_{x\in\gamma}\int_{\mathbb{R}^{d}}a^{+}(x-y)[F(\gamma\cup y)-F(\gamma)]dy.\label{eq:Bolker-Pakala-op}
\end{align}
Here $m>0$ is the mortality rate, $a^{-}$ and $a^{+}$ are competition
and dispersion kernels, respectively. See Section \ref{sec:traveling-waves}
for the conditions on these kernels in the present paper.

We assume that the initial distribution in our model is a probability
measure $\mu_{0}\in\mathcal{M}^{1}(\Gamma)$ and the corresponding
sequence of correlation functions $\varkappa_{0}=(\varkappa_{0}^{(n)})_{n=0}^{\infty}$,
see e.g.\ \cite{KK99}. Then the evolution of the model at time $t>0$
is the measure $\mu_{t}\in\mathcal{M}^{1}(\Gamma)$, and $\varkappa_{t}=(\varkappa_{t}^{(n)})_{n=0}^{\infty}$
its correlation functions. If the evolution of states $(\mu_{t})_{t\geq0}$
is determined by the heuristic Markov generator $L$, then $\mu_{t}$
is the solution of the forward Kolmogorov equation (or Fokker-Plank
equation FPE), 
\begin{equation}
\begin{cases}
\frac{\partial\mu_{t}}{\partial t} & =L^{*}\mu_{t}\\
\mu_{t}|_{t=0} & =\mu_{0},
\end{cases}\label{eq:FPe}
\end{equation}
where $L^{*}$ is the adjoint operator of $L$. In terms of the time-dependent
correlation functions $(\varkappa_{t})_{t\geq0}$ corresponding to
$(\mu_{t})_{t\geq0}$, the FPE may be rewritten as an infinite system
of evolution equations 
\begin{equation}
\begin{cases}
\frac{\partial\varkappa_{t}^{(n)}}{\partial t} & =(L^{\triangle}\varkappa_{t})^{(n)}\\
\varkappa_{t}^{(n)}|_{t=0} & =\varkappa_{0}^{(n)},\qquad n\geq0,
\end{cases}\label{eq:hierarchy}
\end{equation}
where $L^{\triangle}$ is the image of $L^{*}$ in a space of vector-functions
$\varkappa_{t}=(\varkappa_{t}^{(n)})_{n=0}^{\infty}$. The expression
for the operator $L^{\triangle}$ is obtained from the operator $L$
via combinatoric calculations (cf.~\cite{KK99}).

The evolution equation (\ref{eq:hierarchy}) is nothing but a hierarchical
system of equations corresponding to the Markov generator $L$. This
system is the analogue of the BBGKY-hierarchy of the Hamiltonian dynamics,
see \cite{Bo62}.

We are interested in the Vlasov-type scaling of stochastic dynamics
which leads to the so-called kinetic description of the considered
model. In the language of theoretical physics we are dealing with
a mean-field type scaling which is adopted to preserve the spatial
structure. In addition, this scaling will lead to the limiting hierarchy,
which possesses a chaos preservation property. In other words, if
the initial distribution is Poisson (non-homogeneous) then the time
evolution of states will maintain this property. We refer to \cite{FKK10}
for a general approach, other examples, and additional references.

There exists a standard procedure for deriving the Vlasov scaling
from the generator in (\ref{eq:hierarchy}). The specific type of
scaling is dictated by the model in question. The process leading
from $L^{\triangle}$ to the rescaled Vlasov operator $L_{V}^{\triangle}$
produces a non-Markovian generator $L_{V}$ since the positivity-preserving
property fails. Therefore instead of (\ref{eq:FPe}) we consider the
following kinetic FPE, 
\begin{equation}
\begin{cases}
\frac{\partial\mu_{t}}{\partial t} & =L_{V}^{*}\mu_{t}\\
\mu_{t}|_{t=0} & =\mu_{0},
\end{cases}\label{eq:FPe1}
\end{equation}
and observe that if the initial distribution satisfies $\mu_{0}=\pi_{\rho_{0}}$,
then the solution is of the same type, i.e., $\mu_{t}=\pi_{\rho_{t}}$,
$t>0$.

In terms of correlation functions, the kinetic FPE (\ref{eq:FPe1})
gives rise to the following Vlasov-type hierarchical chain (Vlasov
hierarchy) 
\begin{equation}
\begin{cases}
\frac{\partial\varkappa_{t}^{(n)}}{\partial t} & =(L_{V}^{\triangle}\varkappa_{t})^{(n)}\\
\varkappa_{t}^{(n)}|_{t=0} & =\varkappa_{0}^{(n)},\qquad n\geq0.
\end{cases}\label{eq:vlasov_hierarchy}
\end{equation}
This evolution of correlations functions exists in a scale of Banach
spaces, cf.\ \cite{FKK12}.

Let us consider the Lebesgue-Poisson exponent, defined in \eqref{eq:lp-exponent}
\[
\varkappa_{0}(\eta)=e_{\lambda}(\rho_{0},\eta)=\prod_{x\in\eta}\rho_{0}(x),\quad\eta\in\Gamma_{0}
\]
as the initial condition. Such correlation functions correspond to
the Poisson measures $\pi_{\rho_{0}}$ on $\Gamma$ with the density
$\rho_{0}$. The scaling $L_{V}^{\triangle}$ should be such that
the dynamics $\varkappa_{0}\mapsto\varkappa_{t}$ preserves this structure,
or more precisely, $\varkappa_{t}$ should be of the same type 
\begin{equation}
\varkappa_{t}(\eta)=e_{\lambda}(\rho_{t},\eta)=\prod_{x\in\eta}\rho_{t}(x),\quad\eta\in\Gamma_{0}.\label{eq:chaotic_preservation}
\end{equation}
The relation (\ref{eq:chaotic_preservation}) is known as the \emph{chaos
propagation property} of the Vlasov hierarchy. Under certain assumptions
on the mortality $m$ and the kernels $a^{\pm}$, the density $\rho_{t}$
corresponding to the spatial ecologic model, the equation (\ref{eq:chaotic_preservation})
implies, in general a non-linear differential equation for $\rho_{t}$,
$x\in\mathbb{R}^{d}$, 
\begin{equation}
\dfrac{\partial\rho_{t}(x)}{\partial t}=\bigl(a^{+}*\rho_{t}\bigr)(x)-m\rho_{t}(x)-\rho_{t}(x)\bigl(a^{-}*\rho_{t}\bigr)(x),\quad\rho_{t}(x)|_{t=0}=\rho_{0}(x),\label{eq:BPM-density-ode}
\end{equation}
where the initial condition $\rho_{0}$ is a bounded function. Equation
\eqref{eq:BPM-density-ode} is called \emph{Vlasov-type kinetic equation}
for $\rho_{t}$, see \cite{FKK10}, \cite{FKKK15} and references
therein for more details and \cite{Sornette2006} for important applications
of this model in various areas of science.

In general, if one does not start with a Poisson measure, the solution
will leave the space $\mathcal{M}^{1}(\Gamma)$. To have a bigger
class of initial measures, we may consider the cone inside $\mathcal{M}^{1}(\Gamma)$
generated by convex combinations of Poisson measures, denoted by $\mathbb{P}(\Gamma)$. 
\begin{rem}
Below we discuss the concept of a fractional Fokker-Plank equation
and the related fractional statistical dynamics, which is still an
evolution in the space of probability measures $\mathcal{M}^{1}(\Gamma)$
on the configuration space $\Gamma$. The mesoscopic scaling of this
evolutions leads to a fractional kinetic FPE. The subordination principle
provides the representation of the solution to this equation as a
flow of measures that is a transformation of a Poisson flow for the
initial kinetic FPE, see Sections \ref{sec:Solution-Evol-Eq} and
\ref{sec:sub-moving-step} below. 
\end{rem}

Let $X=\{X_{t},\;t\ge0\}$ be the Markov process with generator $L$
given in \eqref{eq:Bolker-Pakala-op}. Denote by $S=\{S_{t},\;t\ge0\}$
the subordinator, independent of $X$, with Laplace exponent $\mathcal{L}(p):=p\mathcal{K}(p)$,
$p\ge0$, that is 
\[
\mathbb{E}\big(e^{-pS_{t}}\big)=e^{-tp\mathcal{K}(p)},\quad p\ge0.
\]
The inverse subordinator $E=\{E_{t},\;t\ge0\}$ (also called the first
hitting time process for the subordinator $S$) is defined by 
\[
E_{t}:=\inf\{s>0:S_{s}>t\},\quad t\ge0
\]
and its density we denote by $\varrho_{t}(\tau)$, that is 
\[
\varrho_{t}(\tau)d\tau=\partial_{\tau}P(E_{t}\le\tau)=\partial_{\tau}P(S_{\tau}\ge t)=-\partial_{\tau}P(S_{\tau}<t).
\]
Then the subordination process $Y_{t}:=X_{E_{t}}$, $t\ge0$ is such
that the one-dimensional distribution $\nu_{t}$ is given by 
\[
\nu_{t}(d\gamma)=\int_{0}^{\infty}\varrho_{t}(\tau)\mu_{\tau}(d\gamma)\,d\tau.
\]
The $t$-Laplace transform of the density $\varrho_{t}(\tau)$ is
equal to 
\[
\int_{0}^{\infty}e^{-ps}\varrho_{s}(\tau)\,ds=\mathcal{K}(p)e^{-\tau p\mathcal{K}(p)}.
\]
Let $k$ be the kernel defined by $\mathcal{K}$ as its Laplace transform
\[
\mathcal{K}(p)=\int_{0}^{\infty}e^{-pt}k(t)\,dt.
\]
With the help of the kernel $k$ we define the general fractional
derivative (GFD) developed in \cite{Kochubei11} which plays a basic
role in this paper 
\begin{equation}
(\mathbb{D}_{t}^{(k)}f)(t):=\frac{d}{dt}\int_{0}^{t}k(t-s)\big(f(s)-f(0)\big)\,ds,\;t>0.\label{eq:GFD}
\end{equation}
In Subsection\ \ref{subsec:GFD} we study in more details the derivative
$\mathbb{D}_{t}^{(k)}$. The natural question about the subordination
process $Y$ is: What type of ``differential'' equation does the
distribution $\nu_{t}$ of $Y_{t}$ satisfies? The answer is as follows:
The distribution $\nu_{t}$ satisfies the following GFD equation 
\[
(\mathbb{D}_{t}^{(k)}\nu_{t})(t)=L\nu_{t},\quad t>0.
\]

As a result we shall consider the fractional Fokker-Plank equation
with the GFD \eqref{eq:GFD} 
\[
\begin{cases}
\mathbb{D}_{t}^{(k)}\mu_{t,k} & =L_{V}^{\triangle}\mu_{t,k}\\
\mu_{t,k}|_{t=0} & =\mu_{0,k}.
\end{cases}
\]
The corresponding evolutions of the correlation functions for the
Vlasov scaling is 
\[
\begin{cases}
\mathbb{D}_{t}^{(k)}\varkappa_{t,k} & =L_{V}^{\triangle}\varkappa_{t,k}\\
\varkappa_{t,k}|_{t=0} & =\varkappa_{0,k}
\end{cases}
\]
which is a non-Markov evolution. We would like to study some properties
of the evolution $\varkappa_{t,k}$. The general subordination principle
gives 
\begin{equation}
\varkappa_{t,k}(\eta)=\int_{0}^{\infty}\varrho_{t}(\tau)\varkappa_{\tau}(\eta)\,d\tau,\quad\eta\in\Gamma_{0},\label{eq:general_subord}
\end{equation}
which is a relation to all orders of the correlation functions. The
kernel $\varrho_{t}$ and its properties are studied in Section\ \ref{sec:Solution-Evol-Eq}
below. In particular, the density of ``particles'' is given by 
\[
\rho_{t}^{k}(x)=\varkappa_{t,k}^{(1)}(x),\quad x\in\mathbb{R}^{d}.
\]
The general subordination principle \eqref{eq:general_subord} gives
\begin{equation}
\rho_{t}^{k}(x)=\int_{0}^{\infty}\varrho_{t}(\tau)\varkappa_{\tau}^{(1)}(x)\,d\tau=\int_{0}^{\infty}\varrho_{t}(\tau)\rho_{\tau}(x)\,d\tau,\quad x\in\mathbb{R}^{d}.\label{eq:subord_density}
\end{equation}
From this representation we should be able to derive an effect of
the fractional derivative onto the evolution of the density, see Sections
\ref{sec:sub-moving-step} and \ref{sec:traveling-waves}. 
\begin{rem}
Certain heuristic motivations in physics are leading to the following
non-linear equation with fractional time derivative: 
\begin{equation}
\mathbb{D}_{t}^{(k)}\rho_{t}(x)=\bigl(a^{+}*\rho_{t}\bigr)(x)-m\rho_{t}(x)-\rho_{t}(x)\bigl(a^{-}*\rho_{t}\bigr)(x).\label{eq:FBPM-density}
\end{equation}
Note that the subordinated density dynamics has no relation with the
solution to this equation. Both evolutions will coincide only in the
case of a linear operator in the right hand side. 
\end{rem}

It is reasonable to study the properties of subordinated flows in
\eqref{eq:subord_density} from a more general point of view, when
the evolution of densities $\rho_{t}(x)$ is not necessarily related
to a particular Vlasov-type kinetic equation, this is realized in
Sections \ref{sec:sub-moving-step} and \ref{sec:traveling-waves}
below.

\subsection{General Fractional Derivative}

\label{subsec:GFD}

\subsubsection{Definitions}

Motivated by the preceding considerations, we recall the general concept
of fractional derivative developed in \cite{Kochubei11} which plays
a basic role in this paper. The basic ingredient of the theory of
evolution equations, \cite{KST2006,Eidelman2004} is to consider,
instead of the first time derivative, the Caputo-Djrba\-shian fractional
derivative of order $\alpha\in(0,1)$ 
\begin{equation}
\big(\mathbb{D}_{t}^{(\alpha)}u\big)(t)=\frac{d}{dt}\int_{0}^{t}k(t-\tau)u(\tau)\,d\tau-k(t)u(0),\quad t>0,\label{eq:Caputo-derivative}
\end{equation}
where 
\[
k(t)=\frac{t^{-\alpha}}{\Gamma(1-\alpha)},\;t>0.
\]
Further details on fractional calculus may be found in \cite{KST2006}
and references therein.

More generally, it is natural to consider differential-convolution
operators 
\begin{equation}
\big(\mathbb{D}_{t}^{(k)}u\big)(t)=\frac{d}{dt}\int_{0}^{t}k(t-\tau)u(\tau)\,d\tau-k(t)u(0),\;t>0,\label{eq:general-derivative}
\end{equation}
where $k\in L_{\mathrm{loc}}^{1}(\mathbb{R}_{+})$ is a nonnegative
kernel. As an example of such operator, we consider the distributed
order derivative $\mathbb{D}_{t}^{(\mu)}$ corresponding to 
\begin{equation}
k(t)=\int_{0}^{1}\frac{t^{-\alpha}}{\Gamma(1-\alpha)}\mu(\alpha)\,d\alpha,\quad t>0,\label{eq:distributed-kernel}
\end{equation}
where $\mu(\alpha)$, $0\le\alpha\le1$ is a positive weight function
on $[0,1]$, see \cite{Atanackovic2009,Daftardar-Gejji2008,Hanyga2007,Kochubei2008,Kochubei2008a,Gorenflo2005,Meerschaert2006}. 
\begin{rem}
\label{rem:general-derivative}

\begin{enumerate}
\item The Caputo-Djrbashian fractional derivative \eqref{eq:Caputo-derivative}
are widely used in physics, see \cite{MK1,Metzler94,Mainardi2010},
for modeling slow relaxation and diffusion processes. In this case
the power-like decay of the mean square displacement of a diffusive
particle appears instead of the classical exponential decay. 
\item Equations with the distributed order operators \eqref{eq:general-derivative}-\eqref{eq:distributed-kernel}
describe ultraslow processes with logarithmic decay, see \cite{Kochubei11,Meerschaert2006}. 
\end{enumerate}
\end{rem}

Considering the general operator (\ref{eq:general-derivative}), it
is natural to investigate the conditions on the kernel $k\in L_{\text{loc}}^{1}(\mathbb{R}_{+})$
such that the operator $\mathbb{D}_{t}^{(k)}$ possess a right inverse
(a kind of a fractional integral) and produce, a kind of a fractional
derivative, equations of evolution type. In particular, it means that 
\begin{enumerate}
\item[(A)] The Cauchy problem 
\begin{equation}
\big(\mathbb{D}_{t}^{(k)}u_{\lambda}\big)(t)=-\lambda u_{\lambda}(t),\quad t>0;\quad u(0)=1,\label{eq:CP1}
\end{equation}
where $\lambda>0$, has a unique solution $u_{\lambda}$, infinitely
differentiable for $t>0$ and completely monotone, $u_{\lambda}\in\mathcal{CM}$,
see Appendix\ \ref{sec:functions-spaces} for this and other classes
of functions in what follows. 
\item[(B)] The Cauchy problem 
\begin{equation}
(\mathbb{D}_{t}^{(k)}w)(t,x)=\Delta w(t,x),\quad t>0,\ x\in\mathbb{R}^{d};\quad w(0,x)=w_{0}(x),\label{eq:CP2}
\end{equation}
where $w_{0}$ is a bounded globally Hölder continuous function, that
is $|w_{0}(x)-w_{0}(y)|\le C|x-y|^{\theta}$, $0<\theta\le1$, for
any $x,y\in\mathbb{R}^{d}$, has a unique bounded solution. In addition,
the equation (\ref{eq:CP2}) possesses a fundamental solution of the
Cauchy problem, a kernel which is a probability density. 
\end{enumerate}
\begin{rem}
\label{rem:CP} 

\begin{enumerate}
\item Gripenberg \cite{Gripenberg1985} has established the well-posedness
of the Cauchy problem for equations with the operator $\mathbb{D}_{t}^{(k)}$
under much weaker assumptions than those in (A) and (B). 
\item When $\mathbb{D}_{t}^{(k)}$ is the Caputo-Djrbashian fractional derivative
$\mathbb{D}_{t}^{(\alpha)}$, $0<\alpha<1$, then $u_{\lambda}(t)=E_{\alpha}(-\lambda t^{\alpha})$
where $E_{\alpha}$ is the Mittag-Leffler function 
\[
E_{\alpha}(z)=\sum\limits _{n=0}^{\infty}\frac{z^{n}}{\Gamma(\alpha n+1)}.
\]
\item The asymptotic properties of $E_{\alpha}$ for real arguments are
given by, see for example \cite{GKMS2014}. 
\[
E_{\alpha}(z)\sim\frac{1}{\alpha}e^{z^{1/\alpha}},\quad\mathrm{as}\;z\to\infty
\]
which resembles the classical case $\alpha=1$, $E_{1}(z)=e^{z}$.
On the other hand 
\begin{equation}
E_{\alpha}(z)\sim-\frac{z^{-1}}{\Gamma(1-\alpha)},\quad\mathrm{as}\;z\to-\infty,\label{eq:ML-asymp}
\end{equation}
so that $u_{\lambda}(t)\sim Ct^{-\alpha}$, $t\to-\infty$. Here and
below $C$ denotes a positive constant which changes from line to
line. This slow decay property is at the origin of a large variety
of applications of fractional differential equations. 
\item In the distributed order case, where $k$ is given by \eqref{eq:general-derivative}-(\ref{eq:distributed-kernel})
with $\mu(0)\ne0$, we have a logarithmic decay, see \cite{Kochubei2008}
\[
u_{\lambda}(t)\sim C(\log t)^{-1},\quad\mathrm{as}\;t\to\infty.
\]
A more general choice of the weight function $\mu$ leads to other
type of decay patterns. 
\end{enumerate}
\end{rem}

The conditions upon $k$ guaranteeing a solution to (A) and (B) were
given in \cite{Kochubei11}. The sufficient conditions are as follows. 
\begin{description}
\item [{(H)}] The Laplace transform 
\begin{equation}
\mathcal{K}(p):=(\mathscr{L}k)(p):=\int_{0}^{\infty}e^{-pt}k(t)\,dt\label{eq:Laplace-k}
\end{equation}
exists and $\mathcal{K}$ belongs to the Stieltjes class $\mathcal{S}$
(or equivalently, the function $\mathcal{L}(p):=p\mathcal{K}(p)$
belongs to the complete Bernstein function class $\mathcal{CBF}$),
and 
\begin{equation}
\mathcal{K}(p)\to\infty,\text{ as \ensuremath{p\to0}};\quad\mathcal{K}(p)\to0,\text{ as \ensuremath{p\to\infty}};\label{eq:H1}
\end{equation}
\begin{equation}
\mathcal{L}(p)\to0,\text{ as \ensuremath{p\to0}};\quad\mathcal{L}(p)\to\infty,\text{ as \ensuremath{p\to\infty}}.\label{eq:H2}
\end{equation}
\end{description}
Under the hypotheses (H), $\mathcal{L}(p)$ and its analytic continuation
admit an integral representation, cf.\ \eqref{eq:representation-CBF}
in Appendix\ \ref{sec:functions-spaces} below (see also \cite{Schilling12})
\begin{equation}
\mathcal{L}(p)=\int_{(0,\infty)}\frac{p}{p+t}\,d\sigma(t)\label{3.16}
\end{equation}
where $\sigma$ is a Borel measure on $[0,\infty)$, such that $\int_{(0,\infty)}(1+t)^{-1}\,d\sigma(t)<\infty$.

\subsubsection{Asymptotic Properties}

As the kernel $k$ and its Laplace transform $\mathcal{K}$ are among
the objects which play a major role in what follows, here we collect
some of its asymptotic properties which depends on the kind of fractional
derivative considered. Two cases are studied, the distributed order
derivative with $k$ given by \eqref{eq:distributed-kernel} and the
general fractional derivative \eqref{eq:general-derivative} for which
$\mathcal{K}$ is a Stieltjes function.

\paragraph{Distributed order derivatives.}

The following two propositions refers to the special case of distributed
order derivative, we refer to \cite{Kochubei2008} for the details
and proofs . 
\begin{prop}[{{{cf.\ \cite[Prop.~2.1]{Kochubei2008}}}}]
\label{prop:distr-order-prop-k}If $\mu\in C^{3}([0,1])$ and $\mu(1)\ne0$,
then 
\begin{align}
k(s) & \sim\frac{1}{s}\frac{1}{(\log s)^{2}}\mu(1),\quad\mathrm{as}\;s\to0,\label{eq:asympt-k-inzero}\\
k'(s) & \sim-\frac{1}{s^{2}}\frac{1}{(\log s)^{2}}\mu(1),\quad\mathrm{as}\;s\to0.\label{eq:asympt-kprime-zero}
\end{align}
\end{prop}

Notice that \eqref{eq:asympt-k-inzero} implies that $k\in L^{1}([0,T])$,
however $k\notin L^{q}([0,T])$ for any $q>1$.

We denote the negative real axis by $\mathbb{R}_{-}:=\{r\in\mathbb{R},\;r\le0\}$. 
\begin{prop}[{{{cf.\ \cite[Prop.~2.2]{Kochubei2008}}}}]
\label{prop:distr-order-prop-K} 

\begin{enumerate}
\item Let $\mu\in C^{2}([0,1])$ be given. If $p\in\mathbb{C}\backslash\mathbb{R}_{-}$
with $|p|\to\infty$, then 
\begin{equation}
\mathcal{K}(p)=\frac{\mu(1)}{\log p}+O\left((\log|p|)^{-2}\right).\label{infty}
\end{equation}
More precisely, if $\mu\in C^{3}([0,1])$, then 
\[
\mathcal{K}(p)=\frac{\mu(1)}{\log p}-\frac{\mu'(1)}{(\log p)^{2}}+O\left((\log|p|)^{-3}\right).
\]
\item Let $\mu\in C([0,1])$ and $\mu(0)\ne0$ be given. If $p\in\mathbb{C}\setminus\mathbb{R}_{-}$,
then 
\begin{equation}
\mathcal{K}(p)\sim\frac{1}{p}\left(\log\frac{1}{p}\right)^{-1}\mu(0),\quad\mathrm{as}\;p\to0.\label{zero}
\end{equation}
\item Let $\mu\in C([0,1])$ be such that $\mu(\alpha)\sim a\alpha^{\lambda}$,
$a>0$, $\lambda>0$. If $p\in\mathbb{C}\setminus\mathbb{R}_{-}$,
then 
\[
\mathcal{K}(p)\sim a\Gamma(1+\lambda)\frac{1}{p}\left(\log\frac{1}{p}\right)^{-1-\lambda},\quad\mathrm{as}\;p\to0.
\]
\end{enumerate}
\end{prop}

\paragraph{Classes of Stieltjes functions.}

Now we turn our attention to the general fractional derivative \eqref{eq:general-derivative}.
Thus, under the assumption (H), the Stieltjes function $\mathcal{K}$
admits the following integral representation \cite{Kochubei11} 
\begin{equation}
\mathcal{K}(p)=\int_{(0,\infty)}\frac{1}{p+t}\,d\sigma(t),\quad p>0,\label{eq:LT-k}
\end{equation}
where $\sigma$ is a Borel measure on $\mathbb{R}_{+}:=[0,\infty)$
such that $\int_{(0,\infty)}(1+t)^{-1}\,d\sigma(t)<\infty$, see also
Definition\ \ref{def:Stieltjes-function}. In other words, $\mathcal{K}$
is the Stieltjes transform of the Borel measure $\sigma$. In addition
we assume that $\sigma\in\mathcal{M}_{\mathrm{abs}}(\mathbb{R}_{+})$,
that is $\sigma$ is absolutely continuous with respect to Lebesgue
measure with a continuous density $\varphi$ on $[0,\infty)$ such
that \eqref{eq:LT-k} turns out 
\begin{equation}
\mathcal{K}(p)=\int_{0}^{\infty}\frac{\varphi(t)}{p+t}\,dt,\label{eq:LT-k-1}
\end{equation}
such that 
\begin{align}
\varphi(t) & \sim Ct^{-\alpha},\quad\mathrm{as}\;t\to\infty,\;0<\alpha<1,\label{wong_infty}\\
\varphi(t) & \sim Ct^{\theta-1},\quad\mathrm{as}\;t\to0,\;0<\theta<1.\label{power_zero}
\end{align}
Then, if $\varphi\in L_{\mathrm{loc}}^{1}([0,\infty))$ it follows
from \cite[Thm.~1, page~299]{Wong2001} (see also \cite{FL}) that
the asymptotic (\ref{wong_infty}) implies the asymptotic for $\mathcal{K}$
\begin{equation}
\mathcal{K}(p)\sim Cp^{-\alpha},\quad\mathrm{as}\;p\to\infty.\label{wong_infty1}
\end{equation}
For the asymptotic of $\mathcal{K}$ at the origin, we have the following
lemma, a special case of a result from \cite[page~326]{Riekstynsh1981}
given there without a proof. 
\begin{lem}
\label{lem:GFD-Lemma}Suppose that 
\begin{equation}
\varphi(t)=Ct^{\theta-1}+\psi(t),\quad0<\theta<1,\label{eq:density}
\end{equation}
where $|\psi(t)|\le Ct^{\theta-1+\delta}$, $0<t\le t_{0}$, and $|\psi(t)|\le Ct^{-\varepsilon}$,
$t>t_{0}$ . Here $0<\delta<1-\theta$ and $\varepsilon>0$. Then
\[
\mathcal{K}(p)\sim Cp^{\theta-1},\quad\mathrm{as}\;p\to0.
\]
\end{lem}

\begin{proof}
It follows from \eqref{eq:LT-k-1} and \eqref{eq:density} that $\mathcal{K}$
is equal to $\mathcal{K}(p)=s_{0}(p)+s_{1}(p)$ where 
\begin{align}
s_{0}(p) & =C\int_{0}^{\infty}\frac{t^{\theta-1}}{t+p}\,dt,\label{eq:integral-1}\\
s_{1}(p) & =C\int_{0}^{t_{0}}\frac{\psi(t)}{t+p}\,dt.\label{eq:integral-2}
\end{align}
The integral in \eqref{eq:integral-1} may be evaluated making the
change of variables $t=p\tau$ and we find that $s_{0}(p)=Cp^{\theta-1}$
(with a different constant $C$ independent of $p$). On the other
hand, $s_{1}(p)$ is estimated by

\[
|s_{1}(p)|\le C\int_{0}^{t_{0}}\frac{t^{\theta-1+\delta}}{t+p}\,dt+C\int_{t_{0}}^{\infty}\frac{t^{-\varepsilon}}{t+p}\,dt\le Cp^{\theta-1+\delta}+C,
\]
Putting together, the required asymptotic follows. 
\end{proof}

\section{Solutions of the Evolution Equations}

\label{sec:Solution-Evol-Eq}Let $L$ be a heuristic Markov generator
defined on functions $u(t,x)$, $t>0$, $x\in\mathbb{R}^{d}$. We
have in mind the Bolker-Pacala model and the related non-linear equation,
see Subsection\ \ref{subsec:Ecological-model} for details. Consider
the evolution equations of the following type 
\begin{align}
\frac{\partial u_{1}(t,x)}{\partial t} & =(Lu_{1})(t,x),\label{eq:EQ-1}\\
(\mathbb{D}_{t}^{(k)}u_{(k)})(t,x) & =(Lu_{(k)})(t,x),\label{eq:EQ-2}
\end{align}
with the same operator $L$ acting in the spatial variables $x$ with
the same initial conditions 
\[
u_{1}(0,x)=\xi(x),\quad u_{(k)}(0,x)=\xi(x).
\]
The solutions of equations \eqref{eq:EQ-1} and \eqref{eq:EQ-2} typically
satisfy the subordination principle \cite{Bazhlekova00}, that is
there exists a nonnegative density kernel function $\varrho_{t}(s)$,
$s,t>0$, such that $\int_{0}^{\infty}\varrho_{t}(s)\,ds=1$ and 
\begin{equation}
u_{(k)}(t,x)=\int_{0}^{\infty}\varrho_{t}(s)u_{1}(s,x)\,ds.\label{5}
\end{equation}

The appropriate notions of the solutions of (\ref{eq:EQ-1}) and (\ref{eq:EQ-2})
depend on the specific setting, they were explained 
\begin{itemize}
\item in \cite{Kochubei11} for the case where $L$ is the Laplace operator
on $\mathbb{R}^{n}$, 
\item in \cite{Bazhlekova00,Baz01,Bazhlekova2015} with abstract semigroup
generators for special classes of kernels $k$, 
\item in \cite{Pruss12} for abstract Volterra equations. 
\end{itemize}
There is also a probabilistic interpretation of the subordination
identities (see, for example, \cite{Kolokoltsov2011,Sato1999}). In
the models of statistical dynamics we deal with a subordination of
measure flows that will give a weak solution to the corresponding
fractional equation.

In the above relation (\ref{5}), the subordination kernel $\varrho_{t}(s)$
does not depend on $L$ and can be found as follows \cite{Kochubei11}.
Consider the function 
\begin{equation}
g(s,p):=\mathcal{K}(p)e^{-s\mathcal{L}(p)},\quad s>0,\;p>0.\label{eq:g}
\end{equation}
The function $p\mapsto e^{-s\mathcal{L}(p)}$ is the composition of
a complete Bernstein and a completely monotone function, then by Theorem\ \ref{thm:chara-Bernstein-ftions}-2
it is a completely monotone function. By Bernstein's theorem (cf.\ Theorem\ \ref{thm:Bernstein}),
for each $s\ge0$, there exists a probability measure $\mu_{s}$ on
$\mathbb{R}_{+}$ such that 
\begin{equation}
e^{-s\mathcal{L}(p)}=\int_{0}^{\infty}e^{-p\tau}\,d\mu_{s}(\tau).\label{eq:Laplace-family}
\end{equation}
The family of measures $\{\mu_{s},\;s>0\}$ is weakly continuous in
$s$. Define 
\begin{equation}
\varrho_{t}(s):=\int\limits _{0}^{t}k(t-\tau)\,d\mu_{s}(\tau).\label{6}
\end{equation}
It follows from \eqref{eq:Laplace-k} and \eqref{eq:Laplace-family}
that the $t$-Laplace transform of $\varrho_{t}(s)$ is equal to $g(s,p$):
\begin{equation}
g(s,p)=\int_{0}^{\infty}e^{-pt}\varrho_{t}(s)\,dt.\label{7}
\end{equation}
It is easy to see from \eqref{eq:g} that 
\[
\int_{0}^{\infty}g(s,p)\,ds=\frac{1}{p}
\]
such that \eqref{7} may be written as 
\[
\int_{0}^{\infty}e^{-pt}\,dt\int_{0}^{\infty}\varrho_{t}(s)\,ds=\frac{1}{p}
\]
which implies the equality 
\[
\int_{0}^{\infty}\varrho_{t}(s)\,ds=1.
\]

\begin{example}[$\alpha$-stable subordinator]
\label{exa:alpha-stable}Let $S$ be a $\alpha$-stable subordinator,
$\alpha\in(0,1)$ with Laplace exponent $\mathcal{L}(p)=p^{\alpha}$
and the corresponding L\'evy measure 
\[
d\sigma(s)=\frac{\alpha}{\Gamma(1-\alpha)}s^{-(\alpha+1)}\,ds.
\]
In this case $\mathcal{K}(p)=p^{\alpha-1}$ and the kernel $k$ is
given by 
\[
k(t)=\frac{t^{-\alpha}}{\Gamma(1-\alpha)}.
\]
The associated general fractional derivative $\mathbb{D}_{t}^{(k)}$
in \eqref{eq:general-derivative} coincides with the Caputo-Djrbashian
fractional derivatives $\mathbb{D}_{t}^{(\alpha)}$, see \eqref{eq:Caputo-derivative}.
As for the density $\varrho_{t}(\tau)$, it follows from Corollary
3.1 in \cite{MSK04} that 
\[
\varrho_{t}(\tau)=\frac{t}{\alpha}\tau^{-1-\nicefrac{1}{\alpha}}g_{\alpha}(t\tau^{-\nicefrac{1}{\alpha}}),
\]
where $g_{\alpha}$ is the density function of $S_{1}$, that is its
Laplace transform is given by 
\[
\tilde{g}_{\alpha}(p)=e^{-p^{\alpha}}.
\]
In addition, it was shown in Proposition\ 1(a) in \cite{Bingham1971},
see also Theorem 4.3 in \cite{Bondesson1996}, that $E_{t}$ has a
Mittag-Leffler distribution, that is 
\begin{equation}
\tilde{\varrho}_{t}(p)=\mathbb{E}(e^{-pE_{t}})=\sum_{n=0}^{\infty}\frac{(-pt^{\alpha})^{n}}{\Gamma(n\alpha+1)}=E_{\alpha}(-pt^{\alpha}).\label{eq:Laplace-density-alpha}
\end{equation}
It follows from the asymptotic of the Mittag-Leffler function $E_{\alpha}$
in \eqref{eq:ML-asymp} that 
\[
\tilde{\varrho}_{t}(p)\sim\frac{C}{t^{\alpha}},\;\mathrm{as}\;t\to\infty.
\]
\end{example}

\section{Subordination of Moving Step Function}

\label{sec:sub-moving-step}Let $u_{0}(t,x)$ be a solution to the
kinetic evolution equation, say for the Bolker-Pacala model from Subsection\ \ref{subsec:Ecological-model}.
The subordinated dynamics is given by 
\begin{equation}
u(t,x)=\int_{0}^{\infty}\varrho_{t}(\tau)u_{0}(\tau,x)\,d\tau.\label{eq:subod-dynamics}
\end{equation}
As a simple example we take a traveling step function (this is a toy
example of a traveling wave in fact) $u_{0}(t,x)=1_{(-\infty,0]}(x-tv)$,
where $x\in\mathbb{R}$, $t,v\in\mathbb{R}_{+}$ and consider the
subordination of the moving step function. In this case 
\[
u(t,x)=\int_{0}^{\infty}\varrho_{t}(\tau)u_{0}(\tau,x)\,d\tau=\int_{x/v}^{\infty}\varrho_{t}(\tau)\,d\tau.
\]
We are interested in studying this dynamics, one possibility is to
study the \emph{Cesaro limit} 
\begin{equation}
M_{t}(u)=\frac{1}{t}\int_{0}^{t}u(\tau,x)\,d\tau,\;\mathrm{as}\;t\to\infty.\label{eq:Cesaro-limit}
\end{equation}
The Cesaro limit \eqref{eq:Cesaro-limit} may be realized in a number
of particular cases related to the fractional derivative considered.
Below we investigate three cases corresponding to the $\alpha$-stable
subordinator, the distributed order derivative and general fractional
derivative with $\mathcal{K}$ fulfilling (H). 
\begin{rem}
The Cesaro limit of the initial step function $u_{0}(\cdot,x)$, for
fixed $x$, is given by 
\[
M_{t}(u_{0})=\frac{1}{t}\int_{0}^{t}u_{0}(\tau,x)\,d\tau=\frac{1}{t}\left(t-\left(0\vee\frac{x}{v}\right)\right)\longrightarrow1,\;t\to\infty.
\]
In addition, for the moving $x(t)=ct^{\beta}$ with $\beta<1$ we
obtain the same asymptotic. Note that the assumption $\beta<1$ is
needed to ensure that $x(t)/t\longrightarrow0$ as $t\to\infty$.

In order to study the Cesaro limit $M_{t}(u)$ in \eqref{eq:Cesaro-limit},
at first we compute the Laplace transform of $u(t,x)$ in $t$ 
\[
\tilde{u}(p,x):=\int_{0}^{\infty}e^{-pt}u(t,x)\,dt=\int_{0}^{\infty}e^{-pt}\int_{x/v}^{\infty}\varrho_{t}(s)\,ds\,dt
\]
using Fubini's theorem and the equality \eqref{7} yields 
\[
\tilde{u}(p,x)=\int_{x/v}^{\infty}g(s,p)\,ds=\int_{x/v}^{\infty}\mathcal{K}(p)e^{-sp\mathcal{K}(p)}\,ds=\frac{1}{p}e^{-\frac{x}{v}p\mathcal{K}(p)}.
\]
The following three cases are distinguished. 
\end{rem}

\begin{enumerate}
\item $\alpha$-stable subordinator. It follows from Example \ref{exa:alpha-stable}
that $\mathcal{K}(p)=p^{\alpha-1}$ which, for fixed $x\in\mathbb{R}$,
implies 
\[
\tilde{u}(p,x)=\frac{1}{p}e^{-\frac{x}{v}p^{\alpha}}=\frac{1}{p}L\left(\frac{1}{p}\right),\;\mathrm{as}\;p\to0,
\]
where $L(x)$, $x>0$ is a slowly varying function, that is 
\[
\lim_{x\to\infty}\frac{L(\lambda x)}{L(x)}=1,\;\mathrm{as}\;x\to\infty,
\]
see also Definition\ \ref{def:RV-functions}-2. It follows from Karamata
Tauberian theorem (cf.\ Theorem\ \ref{thm:Karamata}-(i) with $\rho=1$)
that 
\[
\int_{0}^{t}u(\tau,x)\,d\tau\sim tL(t),\;\mathrm{as}\;t\to\infty
\]
and this implies that the Cesaro limit 
\[
M_{t}(u)=\frac{1}{t}\int_{0}^{t}u(\tau,x)\,d\tau\longrightarrow1,\;\mathrm{as}\;t\to\infty.
\]
The same results holds if ,instead of a fixed $x\in\mathbb{R}$ we
take the moving $x(t)=ct^{\beta}$ with $0<\beta<\alpha<1$. 
\item The distributed order derivative. It follows from the asymptotic at
the origin in \eqref{zero} that 
\[
\tilde{u}(p,x)\sim\frac{1}{p}e^{-\frac{x}{v}\left(\log\frac{1}{p}\right)^{-1}}=:\frac{1}{p}L\left(\frac{1}{p}\right),\;\mathrm{as}\;p\to0,
\]
where $L(p)=e^{-\frac{x}{v}\left(\log\frac{1}{p}\right)^{-1}}$, $p\ge0$
is a slowly varying function. Then an application of the Karamata
Tauberian theorem (cf.\ Theorem\ \ref{thm:Karamata}-(i) with $\rho=1$)
yields 
\[
\int_{0}^{t}u(\tau,x)\,d\tau\sim te^{-\frac{x}{v}(\log t)^{-1}},\;\mathrm{as}\:t\to\infty
\]
which implies the Cesaro limit $M_{t}(u)\longrightarrow1$ as $t\to\infty$,
for any fixed $x\ge0$. For the moving $x(t)=ct^{\beta}$, for any
$\beta>0$, we obtain 
\[
M_{t}u(\cdot,x(t))\longrightarrow0,\;t\to\infty.
\]
Note that the motion of the point $x$ in time with any positive power
turns the Cesaro limit to vanish. 
\item General fractional derivative. If the measure $\sigma$ in \eqref{eq:LT-k}
is absolutely continuous with respect to the Lebesgue measure and
the density $\varphi$ satisfies \eqref{power_zero}, then Lemma\ \ref{lem:GFD-Lemma}
implies that $\mathcal{K}(p)\sim Cp^{\theta-1}$, as $p\to0$ and
\[
\tilde{u}(p,x)\sim\frac{1}{p}e^{-\frac{x}{v}Cp^{\theta}}=\frac{1}{p}L\left(\frac{1}{p}\right),\;\mathrm{as}\;p\to0,
\]
and $L(x)$, $x>0$ is a slowly varying function. Once again by Karamata's
Tauberian theorem (cf.\ Theorem\ \ref{thm:Karamata}-(i) with $\rho=1$)
we obtain the asymptotic for $M_{t}(u)$, namely 
\[
M_{t}(u)\sim e^{-\frac{x}{v}Ct^{-\theta}},\;\mathrm{as}\;t\to\infty
\]
and again we have $M_{t}(u)\longrightarrow1$ as $t\to\infty$. For
the moving $x(t)=ct^{\beta}$ for any $\beta>\theta$, we obtain 
\[
M_{t}u(\cdot,x(t))\sim e^{-Ct^{\beta-\theta}}\longrightarrow0,\;t\to\infty.
\]
The motion of the point $x$ in time with a positive power $\beta$
such that $\beta>\theta$ turns the Cesaro limit to vanish. 
\end{enumerate}

\section{Traveling Waves}

\label{sec:traveling-waves}Now we would like to consider a realistic
dynamics $u_{0}(t,x)$ which is presented by a traveling wave for
the non-local spatial logistic equation. This evolution equation appeared
as the kinetic equation in the Bolker-Pacala ecological model, see
Subsection \ref{subsec:Ecological-model} and \cite{BP1997,FM2004,Finkelshtein2009,FKK10,FKK12,FKKK15}
for more details.

To avoid certain technical details, we will assume the following concrete
relations between the mortality $m$, competition $a^{-}$ and dispersion
$a^{+}$ kernels on the generator $L$ \eqref{eq:Bolker-Pakala-op},
see \cite{FKT2018} for more details. 
\begin{description}
\item [{(A)}] The kernels $a^{\pm}\in L^{\infty}(\mathbb{R}^{d})\cap L^{1}(\mathbb{R}^{d})$
are probability densities, that is 
\[
\int_{\mathbb{R}^{d}}a^{\pm}(y)\,dy=1.
\]
The mortality $0<m<1$. 
\end{description}
\begin{rem}
Under the assumption (A) equation \eqref{eq:BPM-density-ode} has
two constants stationary solutions $\rho_{t}\equiv0$ and $\rho_{t}\equiv1$. 
\end{rem}

A traveling wave $u(t,x)$, $t\ge0$, $x\in\mathbb{R}$ with velocity
$v>0$ is defined by a profile function $\psi:\mathbb{R}\longrightarrow[0,1]$,
that is a continuous monotonically decreasing function such that 
\[
\lim_{x\to-\infty}\psi(x)=1,
\]
\[
\lim_{x\to\infty}\psi(x)=0,
\]
and $u(t,x)=\psi(x-vt)$, $t\ge0$ for almost all $x\in\mathbb{R}$.
For each $\delta>0$ introduce $x_{\delta}\in\mathbb{R}$ as 
\[
\forall x>x_{\delta},\;\;\psi(x)<\delta\quad\mathrm{and}\quad\forall x<-x_{\delta},\;\;\psi(x)>1-\delta.
\]
For a fixed $x\in\mathbb{R}$ the traveling wave $u_{0}(t,x)=\psi(x-vt)$
as a function of $t$ is monotonically increasing and has the following
properties: 
\begin{align}
\psi(x-vt) & <\delta,\quad\forall t<\frac{x-x_{\delta}}{v},\label{eq:travel1}\\
\psi(x-vt) & >1-\delta,\quad\forall t>\frac{x+x_{\delta}}{v},\label{eq:travel2}\\
\psi(x-vt) & \in(0,1),\quad\forall t\in\left]\frac{x-x_{\delta}}{v},\frac{x+x_{\delta}}{v}\right[.\label{eq:travel3}
\end{align}
As a characteristic of this dynamics, we again consider the Cesaro
limit 
\[
M_{t}(u_{0})=\frac{1}{t}\int_{0}^{t}u_{0}(\tau,x)\,d\tau,\;\mathrm{as}\;t\to\infty.
\]
It follows from \eqref{eq:travel1}-\ref{eq:travel3} the following
upper bound for $M_{t}(u_{0})$ 
\[
M_{t}(u_{0})=\frac{1}{t}\int_{0}^{t}\psi(x-v\tau)\,d\tau\leq\frac{1}{t}\left(\delta\frac{x-x_{\delta}}{v}+\frac{2x_{\delta}}{v}+t-\frac{x+x_{\delta}}{v}\right).
\]
On the other hand, a bound from below of $M_{t}(u_{0})$ is obtained
by 
\[
M_{t}(u_{0})\geq\frac{1}{t}\left[\frac{2x_{\delta}}{v}+(1-\delta)\left(t-\frac{x+x_{\delta}}{v}\right)\right].
\]
Putting together, we have 
\[
1-\delta\leq\lim_{t\to\infty}M_{t}(u)\leq1,
\]
due to arbitrary $\delta>0$, the Cesaro limit 
\[
\lim_{t\to\infty}M_{t}(u)=1.
\]

\begin{rem}
Note that for the moving $x(t)=ct^{\beta}$, $0<\beta<1$ the asymptotic
for $M_{t}(u_{0})$ will be the same. 
\end{rem}

Now we will consider the subordinated dynamics 
\begin{equation}
u(t,x)=\int_{0}^{\infty}\varrho_{t}(\tau)u_{0}(\tau,x)d\tau\label{eq:main-caract-travel}
\end{equation}
and study the Cesaro limit 
\begin{equation}
M_{t}(u)=\frac{1}{t}\int_{0}^{t}u(\tau,x)\,d\tau,\;\mathrm{as}\;t\to\infty.\label{eq:mean-value}
\end{equation}
To this end, at first we rewrite $u(t,x)$ as the sum of three terms.
Denoting 
\[
\zeta_{-}:=\frac{x-x_{\delta}}{v},\qquad\zeta_{+}:=\frac{x+x_{\delta}}{v}
\]
we have 
\begin{align}
u(t,x) & =\int_{0}^{\zeta_{-}}\varrho_{t}(\tau)u_{0}(\tau,x)\,d\tau+\int_{\zeta_{-}}^{\zeta_{+}}\varrho_{t}(\tau)u_{0}(\tau,x)\,d\tau+\int_{\zeta_{+}}^{\infty}\varrho_{t}(\tau)u_{0}(\tau,x)\,d\tau\nonumber \\
 & =:I_{1}(t,x)+I_{2}(t,x)+I_{3}(t,x).\label{eq:three-sum}
\end{align}
We study each term separately. 
\begin{description}
\item [{$I_{1}(t,x)$:}] It follows from \eqref{eq:travel1} that 
\[
0\le I_{1}(t,x)\le\delta\int_{0}^{\zeta_{-}}\varrho_{t}(\tau)\,d\tau=\delta\left(1-\int_{\zeta_{-}}^{\infty}\varrho_{t}(\tau)\,d\tau\right).
\]
Therefore, we have 
\[
0\le\frac{1}{t}\int_{0}^{t}I_{1}(s,x)\,ds\le\delta-\frac{\delta}{t}\int_{0}^{t}\int_{\zeta_{-}}^{\infty}\varrho_{t}(\tau)\,d\tau\,ds.
\]
The asymptotic for the integral on the rhs follows as in Section\ \ref{sec:sub-moving-step}
(with $\zeta_{-}$ instead of $\frac{x}{v}$) for all three cases
of $\alpha$-stable subordinator (corresponding to Caputo-Djrbashian
fractional derivatives) distributed order derivative, general fractional
derivative, we have 
\[
\lim_{t\to\infty}\frac{1}{t}\int_{0}^{t}\int_{\zeta_{-}}^{\infty}\varrho_{s}(\tau)\,d\tau\,ds=1
\]
which implies 
\[
\lim_{t\to\infty}\frac{1}{t}\int_{0}^{t}I_{1}(s,x)\,ds=0.
\]
\item [{$I_{2}(t,x):$}] In order to study the behavior in $t$ of 
\[
\frac{1}{t}\int_{0}^{t}I_{2}(s,x)\,ds\leq\frac{1}{t}\int_{0}^{t}\int_{\zeta_{-}}^{\zeta_{+}}\varrho_{s}(\tau)\,d\tau\,ds,\;\mathrm{as}\;t\to\infty
\]
let us at first define 
\[
f(s):=\int_{\zeta_{-}}^{\zeta_{+}}\varrho_{s}(\tau)\,d\tau.
\]
The Laplace transform of $f$ is equal to 
\begin{align*}
\tilde{f}(p) & =\int_{0}^{\infty}e^{-ps}f(s)\,ds\\
 & =\int_{0}^{\infty}e^{-ps}\int_{\zeta_{-}}^{\zeta_{+}}\varrho_{s}(\tau)\,d\tau\,ds\\
 & =\int_{\zeta_{-}}^{\zeta_{+}}g(\tau,p)\,d\tau\\
 & =g_{1}(p)-g_{2}(p),
\end{align*}
where 
\[
g_{1}(p)=\frac{1}{p}e^{-\zeta_{+}p\mathcal{K}(p)},\qquad g_{2}(p)=\frac{1}{p}e^{-\zeta_{-}p\mathcal{K}(p)}.
\]
In all the three cases of fractional derivative we have considered
(e.g., for the general fractional derivative we use our hypothesis
\eqref{eq:H2}) we have $p\mathcal{K}(p)\to0$ as $p\to0$, therefore
$e^{-\zeta_{+}p\mathcal{K}(p)}\to1$ and $e^{-\zeta_{-}p\mathcal{K}(p)}\to1$
as $p\to0$. An application of Karamata's Tauberian theorem, see Theorem\ \ref{thm:Karamata},
yields 
\[
\frac{1}{t}\int_{0}^{t}f(\tau)\,d\tau\sim\exp\left(-\frac{1}{t}\zeta_{+}\mathcal{K}\left(\frac{1}{t}\right)\right)-\exp\left(-\frac{1}{t}\zeta_{-}\mathcal{K}\left(\frac{1}{t}\right)\right),\;\mathrm{as}\;t\to\infty.
\]
Therefore, for the term $I_{2}(t,x)$, we have 
\[
0\le\frac{1}{t}\int_{0}^{t}I_{2}(s,x)\,ds\le\frac{1}{t}\int_{0}^{t}f(s)\,ds\longrightarrow0,\;\mathrm{as}\;t\to\infty
\]
which implies 
\[
\lim_{t\to\infty}\frac{1}{t}\int_{0}^{t}I_{2}(s,x)\,ds=0.
\]
\item [{$I_{3}(t,x):$}] Finally we investigate the Cesaro limit of $I_{3}(t,x)$,
that is 
\[
\frac{1}{t}\int_{0}^{t}I_{3}(\tau,x)\,d\tau,\;\mathrm{as}\;t\to\infty.
\]
It follows from \eqref{eq:travel2} the estimates 
\[
(1-\delta)\frac{1}{t}\int_{0}^{t}\int_{\zeta_{+}}^{\infty}\varrho_{\tau}(s)\,ds\,d\tau\le\frac{1}{t}\int_{0}^{t}I_{3}(\tau,x)\,d\tau\le\frac{1}{t}\int_{0}^{t}\int_{\zeta_{+}}^{\infty}\varrho_{\tau}(s)\,ds\,d\tau.
\]
The integral 
\[
\int_{\zeta_{+}}^{\infty}\varrho_{\tau}(s)\,ds
\]
was studied in Section \ref{sec:sub-moving-step} with $\zeta_{+}=\nicefrac{x}{v}$
and its Cesaro limit, for all the three types of fractional derivatives
considered in Subsection \ref{subsec:GFD}, was shown to be 
\[
\lim_{t\to\infty}\frac{1}{t}\int_{0}^{t}\int_{\zeta_{+}}^{\infty}\varrho_{\tau}(s)\,ds\,d\tau=1.
\]
Hence, we have 
\[
1-\delta\le\lim_{t\to\infty}\frac{1}{t}\int_{0}^{t}I_{3}(\tau,x)\,d\tau\le1.
\]
From the arbitrary of $\delta>0$ we obtain 
\[
\lim_{t\to\infty}\frac{1}{t}\int_{0}^{t}I_{3}(\tau,x)\,d\tau=1.
\]
\end{description}
Putting all together, the Cesaro limit for the subordinated dynamics
by the density $\varrho_{t}(\tau)$ gives 
\[
\lim_{t\to\infty}M_{t}(u)=\lim_{t\to\infty}\frac{1}{t}\int_{0}^{t}u(\tau,x)\,d\tau=1.
\]
This result is true for the three type of fractional derivatives considered
in Subsection \ref{subsec:GFD}.

\appendix

\section{Bernstein, Complete Bernstein and Stieltjes Functions}

\label{sec:functions-spaces}In this appendix we collect certain notions
of functions theory needed throughout the paper. Namely, the classes
of completely monotone, Stieltjes, Bernstein functions and complete
Bernstein functions. They are used in connection with the properties
of the Laplace transform (LT). More details on these classes may be
found in \cite{Schilling12}.

\paragraph{Completely monotone functions. }

The LT (one-sided) of a function $f:[0,\infty)\longrightarrow[0,\infty)$
or a measure $\mu$ on $\mathcal{B}([0,\infty))$ is defined by 
\[
\tilde{f}(p):=(\mathscr{L}f)(p):=\int_{0}^{\infty}e^{-p\tau}f(\tau)\,d\tau\quad\mathrm{or}\quad(\mathscr{L}\mu)(p):=\int_{[0,\infty)}e^{-p\tau}\,d\mu(\tau),
\]
respectively, whenever these integrals converge. It is clear that
$\mathscr{L}u=\mathscr{L}\mu_{u}$ if $d\mu_{u}(\tau)=u(\tau)\,d\tau$.
Finite measures on $[0,\infty)$ are uniquely determined by their
LT. 
\begin{defn}
\label{def:monotone-functions}A $C^{\infty}$-function $\varphi:[0,\infty)\longrightarrow\mathbb{R}$
is called completely monotone if 
\[
(-1)^{n}\varphi^{(n)}(\tau)\ge0,\quad\forall n\in\mathbb{N}_{0}:=\mathbb{N}\cup\{0\},\quad\tau>0.
\]
The family of all completely monotone functions will be denoted by
$\mathcal{CM}$. 
\end{defn}

The function $[0,\infty)\ni\tau\mapsto e^{-\tau t}$, $0\le t<\infty$
is a prime example of a completely monotone function. In fact, any
element $\varphi\in\mathcal{CM}$ can be written as an integral mixture
of this family. This is precisely the contents of the next theorem,
due to Bernstein, on the characterization of the class $\mathcal{CM}$
in terms the LT of positive measures supported on $[0,\infty)$. For
the proof we refer to \cite[Thm.~1.4]{Schilling12}. 
\begin{thm}[Bernstein]
\label{thm:Bernstein}Let $\varphi:(0,\infty)\longrightarrow\mathbb{R}$
be a completely monotone function.

\begin{enumerate}
\item Then there exists a unique measure $\mu$ on $[0,\infty)$ such that
\[
\tilde{\varphi}(p)=\int_{[0,\infty)}e^{-p\tau}\,d\mu(\tau),\quad p>0.
\]
\item Conversely, whenever $\tilde{\varphi}(p)<\infty$, $\forall p>0$,
the function $[0,\infty)\ni p\mapsto\tilde{\varphi}(p)$ is completely
monotone, that is $\varphi$ belongs to the class $\mathcal{CM}$. 
\end{enumerate}
\end{thm}

\begin{rem}
The class $\mathcal{CM}$ of completely monotone functions is easily
seen to be closed under pointwise addition, multiplication and convergence.
However, the composition of elements of the class $\mathcal{CM}$
is, in general, not completely monotone. 
\end{rem}

\paragraph{Stieltjes functions.}

A subclass of the completely monotone functions is the, so called
Stieltjes functions, and they play a central role in the study of
complete Bernstein functions, defined below. 
\begin{defn}
\label{def:Stieltjes-function}A non-negative function $\varphi:(0,\infty)\longrightarrow[0,\infty)$
is a Stieltjes function if it can be written in the form 
\begin{equation}
\varphi(\tau)=\frac{a}{\tau}+b+\int_{(0,\infty)}\frac{1}{\tau+t}\,d\sigma(t),\label{eq:Stieltjes1}
\end{equation}
where $a,b\ge0$ and $\sigma$ is a Borel measure on $(0,\infty)$
such that 
\begin{equation}
\int_{(0,\infty)}(1+t)^{-1}\,d\sigma(t)<\infty.\label{eq:Stieltjes2}
\end{equation}
The family of all Stieltjes functions we denote by $\mathcal{S}$. 
\end{defn}

\begin{rem}
\label{rem:Stieltjes}

\begin{enumerate}
\item The integral in \eqref{eq:Stieltjes1} is called the Stieltjes transform
of the measure $\sigma$. 
\item Using the elementary identity 
\[
(\tau+t)^{-1}=\int_{0}^{\infty}e^{-s(\tau+t)}\,ds
\]
and the Fubini theorem we see that the integral appearing in \eqref{eq:Stieltjes1}
is also a double Laplace transform and $\varphi$ may be written as
\[
\varphi(\tau)=\frac{a}{\tau}+b+\int_{0}^{\infty}e^{-\tau s}h(s)\,ds,
\]
where 
\[
h(s)=\int_{(0,\infty)}e^{-st}\,d\sigma(t)
\]
is a completely monotone function whose LT $\tilde{h}(p)$ exists
for any $p>0$. In particular, we see that $\mathcal{S}\subset\mathcal{CM}$
and $\mathcal{S}$ consists of all $\varphi\in\mathcal{CM}$ such
that its representation measure (from Theorem\ \ref{thm:Bernstein})
has a completely monotone density on $(0,\infty)$, for $\varphi\in\mathcal{S}$
is of the form 
\[
\varphi(p)=(\mathscr{L}a\cdot dt)(p)+\big(\mathscr{L}b\cdot\delta_{0}(dt)\big)(p)+\big(\mathscr{L}(\mathscr{L}\sigma)(t)dt\big)(p).
\]
\end{enumerate}
\end{rem}

\begin{example}
The following are examples of Stieltjes functions for any $\tau,t>0$
\[
\varphi_{1}(\tau)=1,\quad\varphi_{2}(\tau)=\frac{1}{\tau},\quad\varphi_{3}(\tau)=(\tau+t)^{-1},\quad\varphi_{4}(\tau)=\frac{1+t}{\tau+t},
\]
\[
\varphi_{5}(\tau)=\tau^{\alpha-1},\quad\varphi_{6}(\tau)=\frac{1}{\sqrt{\tau}}\arctan\frac{1}{\sqrt{\tau}},\quad\varphi_{7}(\tau)=\frac{1}{\tau}\log(1+\tau).
\]
\end{example}

\paragraph{Bernstein functions.}

Now we introduce the class of Bernstein functions which are closely
related to completely monotone function. Bernstein functions are also
known in probabilistic terms as Laplace exponents. 
\begin{defn}
\label{def:Bernstein-ftion}

\begin{enumerate}
\item A $C^{\infty}$-function $\varphi:(0,\infty)\longrightarrow\mathbb{R}$
is called a Bernstein function if $\varphi(\tau)\ge0$ for all $\tau>0$
and 
\[
(-1)^{n-1}\varphi^{(n)}(\tau)\ge0,\quad\forall n\in\mathbb{N},\;\tau>0.
\]
\item Equivalently, a function $\varphi:(0,\infty)\longrightarrow\mathbb{R}$
is a Bernstein function, if, and only if, it admits the representation
\begin{equation}
\varphi(\tau)=a+b\tau+\int_{(0,\infty)}(1-e^{-\tau t})d\mu(t),\label{eq:Bernstein-ftion}
\end{equation}
where $a,b\ge0$ and $\mu$ is a Borel measure on $(0,\infty)$, called
the L\'evy measure, satisfying 
\begin{equation}
\int_{(0,\infty)}(1\wedge t)\,d\mu(t)<\infty.\label{eq:Levy-measure-BF}
\end{equation}
The L\'evy triplet $(a,b,\mu)$ determines $\varphi$ uniquely and vice
versa. In particular, 
\[
a=\varphi(0+),\qquad b=\lim_{\tau\to\infty}\frac{\varphi(\tau)}{\tau}.
\]
\item The class of Bernstein function will be denoted by $\mathcal{BF}$. 
\end{enumerate}
\end{defn}

The following structural characterization theorem of Bernstein functions
is due to Bochner, see \cite[Thm~3.7]{Schilling12} for the proof. 
\begin{thm}
\label{thm:chara-Bernstein-ftions}Let $\varphi:(0,\infty)\longrightarrow\mathbb{R}$
be a positive function. The following assertions are equivalent.

\begin{enumerate}
\item $\varphi\in\mathcal{BF}$. 
\item $f\circ\varphi\in\mathcal{CM}$, for every $f\in\mathcal{CM}$. 
\item $e^{-\tau\varphi}\in\mathcal{CM}$ for every $\tau>0$. 
\end{enumerate}
\end{thm}

\begin{example}
The following are Bernstein functions 
\begin{multline*}
\varphi_{1}(\tau)=\tau^{\alpha},\;0<\alpha<1,\hspace*{1em}\mathrm{or}\hspace*{1em}\varphi_{2}(\tau)=\frac{\tau}{1+\tau},\hspace*{1em}\mathrm{or}\hspace*{1em}\varphi_{3}(\tau)=\log(1+\tau)
\end{multline*}
which are obtained as an integral mixture of the extremal Bernstein
functions 
\[
e_{0}(\tau)=\tau,\qquad e_{t}(\tau)=\frac{1+t}{t}(1-e^{-\tau t}),\;0<t<\infty,\quad\mathrm{and}\quad e_{\infty}(\tau)=1
\]
by the measures $d\mu(t)=\frac{\alpha}{\Gamma(1-\alpha)}t^{-1-\alpha}$,
$0<\alpha<1$, $d\mu(t)=e^{-t}$ and $d\mu(t)=t^{-1}e^{-t}$, respectively. 
\end{example}

\paragraph{Complete Bernstein functions.}

Finally, we introduce the fourth class of functions, so called complete
Bernstein functions, which are Bernstein functions where the L{\'e}vy
measure $\mu$ in the representation \eqref{eq:Bernstein-ftion} has
a nice density. 
\begin{defn}
\label{def:Complete-Bernstein-ftion}A Bernstein function $\varphi$
is said to be a complete Bernstein function if its L{\'e}vy measure $\mu$
in \eqref{eq:Bernstein-ftion} has a density $\rho$ with respect
to the Lebesgue measure with $\rho\in\mathcal{CM}$. Thus, \eqref{eq:Bernstein-ftion}
takes the form 
\begin{equation}
\varphi(\tau)=a+b\tau+\int_{0}^{\infty}(1-e^{-\tau t})\rho(t)\,dt,\label{eq:CBF}
\end{equation}
such that by \eqref{eq:Levy-measure-BF} we have 
\[
\int_{0}^{\infty}(1\wedge t)\rho(t)\,dt<\infty.
\]
The class of complete Bernstein functions we denote by $\mathcal{CBF}$. 
\end{defn}

The following theorem gives the characterization of complete Bernstein
functions, cf.\ \cite[Thm~6.2]{Schilling12} 
\begin{thm}
\label{thm:caract-cbf}Let $\varphi:(0,\infty)\longrightarrow\mathbb{R}$
be a given non-negative function, then the following expression are
equivalent.

\begin{enumerate}
\item $\varphi\in\mathcal{CBF}$. 
\item The function $(0,\infty)\ni\tau\mapsto\tau^{-1}\varphi(\tau)$ belongs
to $\mathcal{S}$. 
\item There exists a Bernstein function $\psi$ such that 
\[
\varphi(\tau)=\tau^{2}(\mathscr{L}\psi)(\tau),\quad\tau>0.
\]
\item $\varphi$ has an analytic continuation to the upper plane $\mathbb{C}_{>0}:=\{z\in\mathbb{C}\,|\,\Im z>0\}$
such that $\Im\varphi(z)\ge0$ for all $z\in\mathbb{C}_{>0}$ and
the limit $\varphi(0+)=\lim_{\tau\downarrow0}\varphi(\tau)$ exists
and is real. 
\item $\varphi$ has an analytic continuation to the cut complex plane $\mathbb{C}\backslash(\infty,0]$
such that $\Im z\cdot\Im\varphi(z)\ge0$ for all $z\in\mathbb{C}\backslash(\infty,0]$
and the limit $\varphi(0+)=\lim_{\tau\downarrow0}\varphi(\tau)$ exists
and is real. 
\item $\varphi$ has an analytic continuation to $\mathbb{C}_{>0}$ which
is given by 
\begin{equation}
\varphi(z)=a+bz+\int_{(0,\infty)}\frac{z}{z+t}\,d\sigma(t),\label{eq:representation-CBF}
\end{equation}
where $a,b\ge0$ and $\sigma$ is a Borel measure on $(0,\infty)$
satisfying \eqref{eq:Stieltjes2}. 
\end{enumerate}
\end{thm}

\begin{rem}
The constants $a,b$ appearing in both representations \eqref{eq:representation-CBF}
and \eqref{eq:CBF} are the same. The relation between the density
$\rho$ appearing in \eqref{eq:CBF} of the function $\varphi\in\mathcal{CBF}$
and the measure $\sigma$ corresponding to the Stieltjes function
$\psi(\tau)=\tau^{-1}\varphi(\tau)$ is given by 
\[
\rho(\tau)=\int_{(0,\infty)}e^{-\tau t}t\,d\sigma(t).
\]
\end{rem}

The next theorem shows certain nonlinear properties of the class $\mathcal{CBF}$
which gives rise to many applications of this class. Below we use
the shorthand notation $\mathcal{CBF}\circ\mathcal{S}\subset\mathcal{S}$
to indicate that the composition of any $\varphi\in\mathcal{CBF}$
and $f\in\mathcal{S}$ is an element of $\mathcal{S}$, etc. 
\begin{thm}
\label{thm:CBF-properties}

\begin{enumerate}
\item $\varphi\in\mathcal{CBF}\backslash\{0\}$ if, and only if, $\varphi^{*}(\tau):=\tau/\varphi(\tau)$
belongs to $\mathcal{CBF}$. The call $(\varphi,\varphi^{*})$ the
conjugate pair of complete Bernstein functions. 
\item A function $\varphi\not\equiv0$ is a complete Bernstein function
if, and only if, $1/\varphi$ is a non-trivial Stieltjes function. 
\item $\varphi\in\mathcal{CBF}$ if, and only if, $(\tau+\varphi)^{-1}\in\mathcal{S}$
for every $\tau>0$. 
\item $\mathcal{CBF}\circ\mathcal{S}\subset\mathcal{S}$. 
\item $\mathcal{S}\circ\mathcal{CBF}\subset\mathcal{S}$. 
\item $\mathcal{CBF}\circ\mathcal{CBF}\subset\mathcal{CBF}$. 
\item $\mathcal{S}\circ\mathcal{S}\subset\mathcal{CBF}$. 
\end{enumerate}
\end{thm}

We conclude this subsection with some examples of elements in the
class $\mathcal{CBF}$. 
\begin{example}
The following are typical examples of complete Bernstein functions
\begin{multline*}
\varphi_{1}(\tau)=1,\qquad\varphi_{2}(\tau)=\tau,\qquad\mathrm{and}\qquad\varphi_{3}(\tau)=\frac{\tau}{\tau+t},\;0<t<\infty.
\end{multline*}
Using the representation \eqref{eq:representation-CBF} with Stieltjes
measures $\sigma$ of the forms 
\[
\frac{1}{\pi}\sin(\alpha\pi)t^{\alpha-1}dt,\quad1\!\!1_{(0,1)}(t)\frac{dt}{2\sqrt{t}}\quad\text{and \ensuremath{\frac{1}{t}1\!\!1_{(1,\infty)}(t)dt}},
\]
we see that the functions 
\[
\varphi_{4}(\tau)=\tau^{\alpha},\;0<\alpha<1,\quad\varphi_{5}(\tau)=\sqrt{\tau}\arctan\frac{1}{\sqrt{\tau}},\quad\text{and \ensuremath{\varphi_{6}(\tau)=\log(1+\tau)}}
\]
are also complete Bernstein functions. 
\end{example}

\section{The Karamata Tauberian Theorem}

\label{sec:Karamata-thm}Tauberian theorems deals with the deduction
of the asymptotic behavior of functions from a certain class (regular
varying in the original of Karamata \cite{Karamata1933}) from the
asymptotic behavior of their transforms (e.g.\ their Laplace-Stieltjes
transforms). We refer to \cite[Sec.~2.2]{Seneta1976} and \cite{Bingham1987}
for more details and proofs.

Let $A>0$ be given and denote by $\mathcal{F}_{+}(A)$ the class
of positive measurable functions defined on $[A,\infty)$. 
\begin{defn}[Regular and slowly varying functions]
\label{def:RV-functions}Let $f\in\mathcal{F}_{+}(A)$ be given.
We say that $f$ is

\begin{enumerate}
\item \emph{regular varying} (RV) \emph{at infinity} in the sense of Karamata
if the limit 
\[
K_{f}(\lambda):=\lim_{x\to\infty}\frac{f(\lambda x)}{f(x)}
\]
exists and is finite for all $\lambda>0$. 
\item \emph{slowly varying} (SV) if 
\[
K_{f}(\lambda)=1,\quad\forall\lambda>0.
\]
\end{enumerate}
\end{defn}

\begin{prop}
Let $f\in\mathcal{F}_{+}(A)$ be a RV function. 

\begin{enumerate}
\item Then there is a real number $\rho$ (called index of the function
$f$) such that 
\[
K_{f}(\lambda)=\lambda^{\rho},\quad\lambda>0.
\]
The index $\rho=0$ characterizes the SV functions, that is $K_{f}(\lambda)=\lambda^{0}=1$\textcompwordmark . 
\item Any RV function $f$ of index $\rho$ is represented as 
\[
f(x)=x^{\rho}l(x),\quad\forall x\ge A,
\]
where $l$ is a corresponding SV function. 
\end{enumerate}
\end{prop}

We say that the functions $f$ and $g$ are \emph{asymptotically equivalent
at infinity}, and denote $f\sim g$ as $x\to\infty$, meaning that
\[
\lim_{x\to\infty}\frac{f(x)}{g(x)}=1.
\]

\begin{thm}[Karamata's Tauberian Theorem]
\label{thm:Karamata}Let $U:[0,\infty)\longrightarrow\mathbb{R}$
be a monotone non-decreasing function such that 
\[
w(x):=\int_{0}^{\infty}e^{-xs}\,dU(s)<\infty,\quad\forall x>0.
\]
Then, if $\rho\ge0$ and $L$ is a slowly varying function 

\begin{enumerate}
\item[(i)] $w(x)=x^{-\rho}L\left(\frac{1}{x}\right)$ as $x\to0^{+}$ $\Longrightarrow$
$U(x)\sim x^{\rho}L(x)/\Gamma(\rho+1)$ as $x\to\infty$; 
\item[(ii)] $w(x)=x^{-\rho}L(x)$ as $x\to\infty$ $\Longrightarrow$ $U(x)\sim x^{\rho}L\left(\frac{1}{x}\right)/\Gamma(\rho+1)$
as $x\to0^{+}$. 
\end{enumerate}
\end{thm}

\subsection*{Founding}
Jos{\'e} L. da Silva is a member
of the Centro de Investiga{\c c\~a}o em Matem{\'a}tica e Aplica{\c c\~o}es
(CIMA), Universidade da Madeira, a research centre supported with
Portuguese funds by FCT (Funda{\c c\~a}o para a Ci{\^e}ncia e a
Tecnologia, Portugal) through the Project UID/MAT/04674/2013.

\end{document}